\documentclass[a4paper,12pt]{article}
\usepackage{amssymb,amsmath,amsfonts,amsthm}
\usepackage[utf8]{inputenc} 
\usepackage[english]{babel}
\usepackage[pdftex,unicode]{hyperref}
\usepackage{indentfirst} 
\usepackage{ytableau}
\usepackage{young}
  
\newtheorem{lemma}{Lemma}
\newtheorem{rem}{Remark}
\newtheorem{prb}{Problem}
\newtheorem{theorem}{Theorem}
\newtheorem{definition}{Definition}
\newtheorem{corollary}{Corollary}
\newtheorem{prop}{Proposition}
\usepackage{enumerate}
\usepackage{adjustbox,tabu}

\author{Nanying Yang\thanks{The first author was supported by NNSF grant of China (No. 11301227)}\quad and Alexey Staroletov}
\title{The minimal polynomials of powers of cycles in the ordinary representations of symmetric and alternating groups}

\date{\vspace{-5ex}}

\begin{document}

\newcommand{\Addresses}{{
		\bigskip
		\footnotesize
		
		N.~Yang, \textsc{School of Science, Jiangnan University, Wuxi, 214122, P.R.~China;}\par\nopagebreak
		\textit{E-mail address: } \texttt{yangny@jiangnan.edu.cn}
		
		\medskip
		
		A.~Staroletov, \textsc{Sobolev Institute of Mathematics, Novosibirsk 630090, Russia;}\par\nopagebreak
		\textsc{Novosibirsk State University, Novosibirsk 630090, Russia;}\par\nopagebreak
		\textit{E-mail address: } \texttt{staroletov@math.nsc.ru}
}}

\maketitle	

\begin{abstract}
 Denote the alternating and symmetric groups of degree $n$ by $A_n$ and $S_n$ respectively.
 Consider a permutation $\sigma\in S_n$ all of whose nontrivial cycles are of the same length.
 We find the minimal polynomials of $\sigma$ in the ordinary irreducible representations of $A_n$ and $S_n$.
\end{abstract}

\section{Introduction}
The question of finding the minimal polynomials of elements in representations for
a given group goes back to the famous work of P.\,Hall and G.\,Higman \cite{HallHig}.

Assume that $\mathbb{F}$ is an algebraically closed field and $G$ is a finite group.
Consider some irreducible representation $\rho$ of $G$ over $\mathbb{F}$.
For $g\in G$ denote by $deg(\rho(g))$ the degree of the minimal polynomial of the matrix $\rho(g)$
and by $o(g)$ the order of $g$ modulo $Z(G)$. 
The general problem was formulated in \cite{TiepZal08} as follows.

\begin{prb} Determine all possible values for $\deg(\rho(g))$, and if possible, all triples
$(G, \rho, g)$ with $deg(\rho(g))<o(g)$, in the first instance under the condition that $o(g)$
is a $p$-power.
\end{prb}

There are plenty of publications in this area and we mention only the results on symmetric and alternating groups of degree $n$, which will be denoted by $A_n$ and $S_n$ respectively. The minimal polynomials of prime order elements of $A_n$ and $S_n$ in the ordinary or projective representations were found in \cite{Zal96}.
For the algebraically closed fields of positive characteristic $p$, Kleshchev and Zalesski \cite{KleshZal04} described the minimal polynomials of order $p$ elements in the irreducible representations of covering groups of $A_n$.

In this paper we describe the minimal polynomials of some elements in the ordinary representations of $A_n$ and $S_n$. Namely, we consider permutations with the cycle decomposition consisting of cycles of a fixed length and cycles of length one. This set includes the set of permutations of prime order and consists of powers of single cycles. 

Recall that a partition of $n$ is a nonincreasing finite sequence of
positive integers $(\lambda_1, \lambda_2,\ldots, \lambda_k)$ whose sum equals~$n$. It is well-known that there exists a one-to-one correspondence between the set of equivalence classes 
of ordinary irreducible representations of $S_n$ and the set of partitions of~$n$. 
The symmetric group $S_n$ has several named representations. The
{\it alternating representation} is the representation of degree one mapping each permutation to
its sign. Furthermore,  if $V$ is a vector space of dimension $n$ over a field of characteristic zero with base $e_1, e_2,\ldots, e_n$ then $S_n$ acts naturally on $V$: for every $\sigma\in~S_n$ and  $i\in\{1,\ldots,n\}$ we have $\sigma(e_i)=e_{\sigma(i)}$. There are two irreducible constituents in this representation: the line 
$l=\langle e_1+e_2+\ldots+e_n\rangle$
and its orthogonal complement $l^\perp=\langle e_1-e_i~|~ 2\leq i\leq n \rangle$. This action of $S_n$ on $l^\perp$ is called {\it the standard irreducible representation} of $S_n$. It corresponds to the partition $(n-1,1)$.
Every irreducible representation corresponding to the partition $(2,1^{n-1})=(2,1,\ldots,1)$ is called {\it the associated representation with the standard representation}.

We say that a permutation $\sigma\in S_n$ is of shape $[a_1^{b_1}a_2^{b_2}\ldots a_k^{b_k}]$, where $a_i$ are distinct integers, 
if the cycle decomposition of $\sigma$ is comprised of $b_i$ cycles of length $a_i$ for $i=1\ldots k$. 
Denote the field of complex numbers by $\mathbb{C}$. The main result of this paper is the following.

\begin{theorem}\label{th:1}
Given positive integers $n$, $r$ and $m$ with $n\geq3$, $r\geq2$ and $rm\leq n$, assume that $\sigma\in S_n$ is of shape $[r^m1^{n-rm}]$, i.e. a product of $m$ cycles of length $r$, and $\rho:S_n\rightarrow GL(V)$ is a nontrivial irreducible representation of $S_n$
over $\mathbb{C}$. Denote the minimal polynomial of $\rho(\sigma)$ by $\mu_{\rho(\sigma)}(x)$.
Then $\mu_{\rho(\sigma)}(x)\neq x^r-1$ if and only if one of the following statements holds.
\begin{enumerate}[(i)]
\item We have $r=n$, $m=1$, and $\rho$ is the standard representation; in this case $\mu_{\rho(\sigma)}(x)=~\frac{x^n-1}{x-1}$.
\item We have $r=n$, $m=1$, and $\rho$ is associated with the standard representation; in this case $\mu_{\rho(\sigma)}(x)=\frac{x^n-1}{x+(-1)^n}$.
\item We have $r=n=6$, $m=1$, and $\rho$ corresponds to $\lambda=(3,3)$; in this
case $\mu_{\rho(\sigma)}(x)=\frac{x^6-1}{x^2+x+1}$.
\item We have $r=n=6$, $m=1$, and $\rho$ corresponds to $\lambda=(2,2,2)$; in this
case $\mu_{\rho(\sigma)}(x)=\frac{x^6-1}{x^2-x+1}$.
\item We have $n=4$, $\rho$ corresponds to $\lambda=(2,2)$; and $(r,m,\mu_{\rho(\sigma)}(x))\in\{(4,1,x^2-1),(3,1,x^2+x+1),(2,2,x-1)\}$.
\item $\rho$ is the alternating representation and $\mu_{\rho(\sigma)}(x)=x-(-1)^{\operatorname{sgn}(\sigma)}$.

\end{enumerate}
\end{theorem}

Recall that {\it the standard representation} of $A_n$ is the restriction of the standard representation of $S_n$ to $A_n$.
As a corollary of Theorem~\ref{th:1} we prove the following statement.

\begin{corollary}\label{cor:1}
Given positive integers $n\geq5$, $r\geq2$ and $rm\leq n$, assume that $\sigma\in A_n$ is of shape $[r^m1^{n-rm}]$, and $\rho:A_n\rightarrow GL(V)$ is a nontrivial irreducible representation of $A_n$
over $\mathbb{C}$. Denote the minimal polynomial of $\rho(\sigma)$ by $\mu_{\rho(\sigma)}(x)$.

Then $\mu_{\rho(\sigma)}(x)\neq x^r-1$ if and only if one of the following statements holds.
\begin{enumerate}[(i)]
\item We have $r=n$ is odd, $m=1$, and $\rho$ is the standard representation; in this case $\mu_{\rho(\sigma)}(x)=~\frac{x^n-1}{x-1}$.
\item We have $r=n=5$, $m=1$, and $\rho$ corresponds to $\lambda=(3,1,1)$;
in this case $\mu_{\rho(\sigma)}(x)=~\frac{x^5-1}{(x-\eta)(x-\overline{\eta})}$, where $\eta$
is a primitive 5th root of unity.
\end{enumerate}
\end{corollary}

There are a few other examples of permutations $\sigma\in S_n$ such that 1 is not an eigenvalue of the image of $\sigma$ in an irreducible representation of $S_n$.
We list them in the following proposition.

\begin{prop}\label{pr:1} For an integer $n\geq3$, consider the irreducible representation $\rho:S_n\rightarrow GL(V)$ corresponding to a partition $\lambda$ of $n$. 
Assume that a pair $(\lambda,\sigma)$ is one of the following: 

\begin{enumerate}[(i)]
\item $((2^2,1^{n-4}), [(n-2)^12^1])$, where $n$ is odd,
\item or $\lambda=(1^n)$ and $\sigma$ is odd,
\item or $\sigma$ is $n$-cycle and either $\lambda=(n-1,1)$ or $n$ is odd and $\lambda=(2,1^{n-2}),$
\item or $((2^3), [3^12^11^1])$, $((4^2), [5^13^1])$, $((2^4), [5^13^1])$, $((2^5), [5^13^12^1]).$
\end{enumerate}

\noindent Then $1$ is not an eigenvalue of $\rho(\sigma)$.
\end{prop}

\begin{rem} Using GAP~$\cite{GAP}$, we verify that for $3\leq n\leq 20$ Theorem~\ref{th:1} and Proposition~\ref{pr:1} include all examples 
of permutations $\sigma\in S_n$ and corresponding irreducible representations such that $1$ is not an eigenvalue of $\rho(\sigma)$.
\end{rem}

This paper is organized as follows. In Section~2 we recall notation and basic facts about ordinary representations of symmetric and alternating groups. In Section~3 we formulate and prove some auxiliary results used for proving Theorem~\ref{th:1} and Corollary~\ref{cor:1}. Sections 4--6 are devoted to the proofs of the main results.

\section{Ordinary representations: notation and basic facts}

Our notation for the ordinary representations of $S_n$ and $A_n$ follows \cite{James}. Let $n$ be a positive integer. We write $\lambda\vdash n$ for a partition $\lambda=(\lambda_1,\lambda_2,\ldots,\lambda_k)$, that is, a sequence of positive integers  $\lambda_i$ with $\lambda_1\geq \lambda_2\geq\ldots\geq\lambda_k$ and $\sum_{i=1}^k\lambda_i=n$. If $\lambda\vdash n$ then by $T^\lambda$ we denote the Young diagram of shape~$\lambda$, consisting of $k$ rows with $\lambda_1$, $\lambda_2$,$\ldots$, $\lambda_k$  
square boxes, respectively (see Figure~\ref{f:young}). 
\ytableausetup{centertableaux}

\begin{figure}[!h]
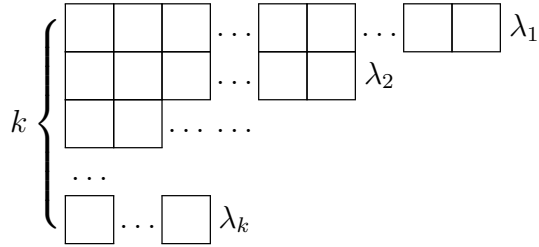

	\centering
	\begin{tabular}{r@{}l}
		\raisebox{-0.0ex}{$k\left\{\vphantom{\begin{array}{c}~\\[10ex] ~
				\end{array}}\right.$} 
		&
		\begin{ytableau}
			~  &       &      & \none[\dots]  &  &  & \none[\dots] & & &\none[\lambda_1] \\
			~  &       &       & \none[\dots]  &  & & \none[\lambda_2] \\
			~               &       &  \none[\dots]  & \none[\dots]  & \none   \\
			\none[\dots]    & \none & \none & \none         & \none \\
			~  & \none[\dots] &     & \none[\lambda_k] & \none & \none         & \none \\
			
		\end{ytableau}\\[-1.5ex]

	\end{tabular}
	\caption{The diagram $T^\lambda = (\lambda_1,\lambda_2,\ldots,\lambda_k)$}\label{f:young}
\end{figure}

\begin{definition} The column lengths of $T^{\lambda}$ form a partition of $n$
which is called {\it the partition associated with} $\lambda$ and denoted by $\lambda'$.
In other words, $\lambda'=(\lambda'_1,\lambda'_2,\ldots,\lambda'_{\lambda_1})$, where $\lambda'_i=\sum_{j, \lambda_j\geq i}1$. If $\lambda=\lambda'$ then $\lambda$
is called {\it self-associated}.
\end{definition}

We naturally numerate the boxes of $T^\lambda$ by the pairs $(i,j)$
meaning $j$ in row $i$.

\begin{definition} The set of boxes to the right or below box $(i,j)$ in $T^{\lambda}$ together with the latter forms the {\it hook} $H^{\lambda}_{ij}$ (see Figure~\ref{f:hook}).
The number of boxes in $H^{\lambda}_{ij}$ is denoted by $h^{\lambda}_{ij}$
and called {\it the length} of $H^{\lambda}_{ij}$. Using the elements of $\lambda'$,
we can express $h^{\lambda}_{ij}$ as $\lambda_i-j+\lambda'_j-i+1$.
\end{definition}

\ytableausetup{centertableaux}

\begin{figure}[!h]
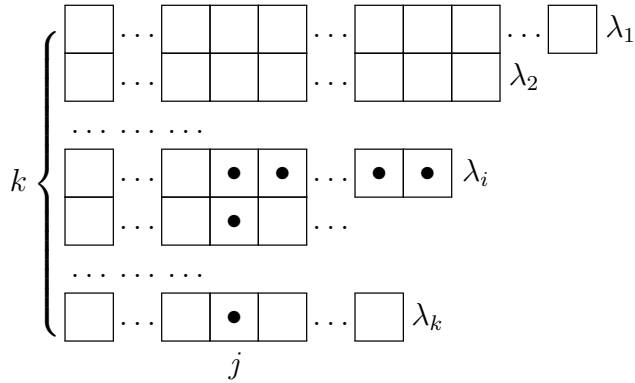

	\centering
	\begin{tabular}{r@{}l}
		\raisebox{1.0ex}{$k\left\{\vphantom{\begin{array}{c}~\\[17ex] ~
				\end{array}}\right.$} 
		& 
		\begin{ytableau}
			~  &  \none[\dots] &  &     &  & \none[\dots]  & &  &  & \none[\dots]  &   & \none[\lambda_1] \\
			~  &  \none[\dots] &  &    &  & \none[\dots]  & &  & & \none[\lambda_2] \\
			\none[\dots]  & \none[\dots] & \none[\dots] & \none & \none         & \none \\
			~  &  \none[\dots] &  &  \bullet  & \bullet & \none[\dots]  & \bullet & \bullet & \none[\lambda_i] \\
			~  &  \none[\dots] &  &  \bullet  &  &\none[\dots]  & \none         & \none   \\
			\none[\dots]  & \none[\dots] & \none[\dots] & \none & \none         & \none \\
			~  &  \none[\dots] &  & \bullet & & \none[\dots]  &  & \none[\lambda_k] \\
			   \none & \none & \none & \none[\emph{j}] \\
		\end{ytableau}\\[-1.5ex]
	\end{tabular}
	\caption{The hook $H^\lambda_{ij}$}\label{f:hook}
\end{figure}

\begin{definition} Let $H^{\lambda}_{ij}$ be a hook in $T^{\lambda}$. The number $l_{ij}=\lambda'_j-i$ is called the leg length of $H^{\lambda}_{ij}$. The part of $T^{\lambda}$ consisting of the boxes on the rim between the lower left and the upper right boxes of $H^{\lambda}_{ij}$ is denoted by $R^{\lambda}_{ij}$ and called a rim hook (see Figure~\ref{f:r-hook}).	
\end{definition}

\begin{figure}[!h]
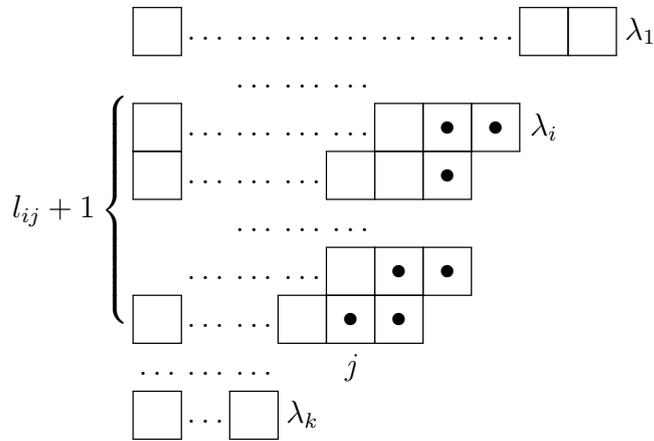

	\centering
	\begin{tabular}{r@{}l}
		\raisebox{1.0ex}{$l_{ij}+1\left\{\vphantom{\begin{array}{c}~\\[11ex] ~
				\end{array}}\right.$} 
		& 
		\begin{ytableau}
			~  &  \none[\dots] & \none[\dots] & \none[\dots] & \none[\dots] & \none[\dots] & \none[\dots] & \none[\dots] &  &  & \none[\lambda_1] \\
			\none & \none &  \none[\dots]  & \none[\dots] & \none[\dots] & \none & \none         & \none \\
			~  &  \none[\dots] & \none[\dots] & \none[\dots]  & \none[\dots]  & & \bullet & \bullet & \none[\lambda_i] \\
			~  &  \none[\dots] & \none[\dots] & \none[\dots]  &  & & \bullet   & \none   \\
			\none & \none &  \none[\dots]  & \none[\dots] & \none[\dots] & \none & \none         & \none \\
			\none & \none[\dots]  &  \none[\dots] &  \none[\dots] &  & \bullet  & \bullet  \\
			~  & \none[\dots] & \none[\dots]  &  &  \bullet & \bullet  \\
			\none[\dots]  & \none[\dots] & \none[\dots] & \none &\none[\emph{j}] & \none         & \none \\
			~  &  \none[\dots] &  & \none[\lambda_k] \\
			
		\end{ytableau}\\[-1.5ex]
	\end{tabular}
	\caption{The rim hook $R^\lambda_{ij}$}\label{f:r-hook}
\end{figure}

It is well-known that there exists a bijection between the set of equivalence classes of ordinary irreducible representations of $S_n$ and the set of partitions of~$n$. 
The equivalence class of irreducible representations corresponding to a partition $\lambda\vdash n$ is denoted by $[\lambda]$ and its character, by $\chi^{\lambda}$. 
By $[\lambda]\downarrow A_n$ we denote the restriction of $[\lambda]$ to $A_n$. 

We collect some basic facts about the ordinary irreducible representations and characters of symmetric and alternating groups in the following three propositions.
\begin{prop}\label{l:rep_Sn} For every integer $n>1$, the following statements hold.
	
\begin{enumerate}[(i)] 
	\item The splitting field for $S_n$ is the field of rationals $\mathbb{Q}$ and the values of all characters of $S_n$ are integers.
	
	\item We have the hook-formula $\chi^\lambda(1)=\cfrac{n!}{\prod_{(i,j)\in T^{\lambda}}h^{\lambda}_{ij}}$.
	
	\item If $\sigma\in S_n$ then $\chi^{\lambda'}(\sigma)=\operatorname{sgn}(\sigma)\cdot\chi^{\lambda}(\sigma)$.
\end{enumerate}
	
\end{prop}	

\begin{prop}$\cite[Theorem~2.5.7]{James}\label{p:reps_of_An}$
Suppose that $\lambda$ is a partition of $n>1$. Then the following statements hold.
\begin{enumerate}[(i)]
	\item If $\lambda\neq\lambda'$ then $[\lambda]\downarrow A_n=[\lambda']\downarrow A_n$ is irreducible.
	
	\item If $\lambda=\lambda'$ then 
	$[\lambda]\downarrow A_n=[\lambda']\downarrow A_n$ splits into two irreducible and conjugate representations $[\lambda]^{\pm}$, i.e., $[\lambda]^{+(12)}$, defined by $[\lambda]^{+(12)}((12)\sigma(12)):=[\lambda]^{+}(\sigma)$, where $\sigma\in A_n$, 
	is equivalent to $[\lambda]^{-}$.
	
	\item The set $\{[\lambda]\downarrow A_n~|~\lambda\neq\lambda'  \}\cup\{[\lambda]^{\pm}~|~\lambda=\lambda'\vdash n\} $ is a complete system of equivalence classes of ordinary irreducible representations  of $A_n$.
	
\end{enumerate}	
\end{prop}

\begin{prop}$\cite[Theorem~2.5.13]{James}$\label{p:chars_of_An} Suppose that $\lambda=\lambda'$ is a self-associated partition of $n>1$ and denote by $\chi^{\lambda\pm}$ the characters of $[\lambda]^{\pm}$. Assume that the main diagonal of $T^{\lambda}$ has $l$ boxes and denote by $a_1$, $a_2$,$\ldots$, $a_l$ the lengths of hooks along the main diagonal. 
	Then the following statements hold.
\begin{enumerate}[(i)]
	\item If $\sigma\in A_n$ and $\sigma$ is not of shape $[a_1^1a_2^1\ldots a_l^1]$ then $\chi^{\lambda\pm}(\sigma)=\cfrac{1}{2}\chi^{\lambda}(\sigma)$ and this number is an integer.
	\item If $\sigma\in A_n$ and $\sigma$ is of shape $[a_1^1a_2^1\ldots a_l^1]$ then
	$$\chi^{\lambda\pm}(\sigma)\in\big\{\cfrac{1}{2}\big(\chi^\lambda(\sigma)+\sqrt{\chi^\lambda(\sigma)\cdot\Pi_{i}a_i}\big), \cfrac{1}{2}\big(\chi^\lambda(\sigma)-\sqrt{\chi^\lambda(\sigma)\cdot\Pi_{i}a_i}\big)\big\}.$$
\end{enumerate}	
\end{prop}

Finally we formulate three famous rules that help to evaluate characters on the elements of $S_n$.

\begin{prop}$\cite[Theorem~2.4.3]{James}$(The Branching Law) If $\lambda=(\lambda_1,\lambda_2,\ldots,)$ is a partition of $n$ then for the restriction of $[\lambda]$ to the stabilizer $S_{n-1}$ of the point~$n$ we have
$$[\lambda]\downarrow S_{n-1}=\sum_{i:\lambda_i>\lambda_{i+1}}[\lambda^{i-}], $$
where $\lambda^{i-}$ is a partition of $n-1$ equal to $(\lambda_1,\ldots,\lambda_{i-1},\lambda_i-1,\lambda_{i+1},\ldots)$.
\end{prop}

\begin{prop}\cite[2.4.7]{James}(The Murnaghan--Nakayama formula)\label{NM-formula}
Let $n$	and $k$ be integers with $k\leq n$. Consider $\sigma\in S_n$
such that $\sigma=\pi\cdot\rho$, where $\pi$ fixes the symbols $1,2,\ldots,n-k$ and $\rho$
is the cycle $(n-k+1,n-k+2,\ldots,n)$ of length $k$. Then
$$\chi^{\lambda}(\sigma)=\sum_{i,j:h^{\lambda}_{ij}=k}(-1)^{l^{\lambda}_{ij}}
\chi^{\lambda\setminus{R_{ij}}}(\pi),$$
where $\lambda\setminus{R_{ij}}$ denotes the partition obtained from $\lambda$ by removing the rim hook~ 
$R_{ij}$.
\end{prop}

\begin{prop}$\cite[4.10]{FulHar}$(Frobenius's Formula)\label{p:frob_formula}
Let $n$ be a positive integer and $\lambda=(\lambda_1,\lambda_2,\ldots,\lambda_k)$ be a partition of $n$. Put $l_1=\lambda_1+k-1$, $l_2=\lambda_1+k-2$,$\ldots$, $l_k=\lambda_k$. Introduce the power sum polynomials  $P_j(x)=x_1^j+x_2^j+\ldots+x_k^j$ and the
discriminant polynomial $\Delta(x)=\prod_{i<j}(x_i-x_j)$.
If $\sigma$ is of shape $[1^{i_1}2^{i_2}\ldots n^{i_n}]$ then
$$\chi^\lambda(\sigma)=[\Delta(x)\cdot\prod_j P_j(x)^{i_j}]_{(l_1,l_2,\ldots,l_k)},$$
where $[f(x)]_{(l_1,l_2,\ldots,l_k)}$ stands for the coefficient of $x_1^{l_1}x_2^{l_2}\ldots x_k^{l_k}$ in $f(x)$.
\end{prop}

\section{Preliminaries}

Here is our main tool to estimate character values on permutations.
\begin{lemma}$\cite[Theorem~1.1]{FominLulov}$\label{l:Fom_Lul}
 Suppose that $n=rm$, where $r$ and $m$ are integers, and $\sigma$ is of shape $[r^m]$.
Assume that $\rho:S_n\rightarrow V$	is an ordinary irreducible representation of $S_n$
and $\chi_V$ is the character of $\rho$. Then $$|\chi_V(\sigma)|\leq \frac{m!\cdot r^m}{(n!)^{1/r}}(\chi_V(1))^{1/r}.$$

\end{lemma}

To estimate factorials, we use the following well-known result due to H.\,Robbins.
\begin{lemma}$\cite{Rob55}$\label{l:Stirling} Let $n\geq1$ be an integer. Then
$$\sqrt{2\pi}\cdot n^{n+1/2}\cdot e^{-n}\cdot e^{1/(12n+1)}<n!<\sqrt{2\pi}\cdot n^{n+1/2}\cdot e^{-n}\cdot e^{1/(12n)}.$$
\end{lemma}

We apply these inequalities to the conclusion of Lemma~\ref{l:Fom_Lul}.

\begin{lemma}\label{l:ch_estimate} Suppose that $r$ and $m$ are integers, $n=rm$, and $\sigma$ is a permutation of shape $[r^m]$.
For every irreducible character $\chi^\lambda$ of $S_n$ we have
$$\frac{|\chi^{\lambda}(\sigma)|}{(\chi^{\lambda}(1))^{1/r}}<(2\pi)^{\frac{r-1}{2r}}r^\frac{-1}{2}n^{\frac{r-1}{2r}}e^\frac{1}{12m}.$$
\end{lemma}
\begin{proof} Lemma~\ref{l:Fom_Lul} implies that
$$\frac{|\chi^{\lambda}(\sigma)|}{(\chi^{\lambda}(1))^{1/r}}\leq \frac{m!\cdot{r^m}}{(n!)^{1/r}}.$$
Using Lemma~\ref{l:Stirling}, we obtain $$\cfrac{1}{n!}<\cfrac{1}{\sqrt{2\pi}\cdot n^{n+1/2}\cdot e^{-n}\cdot e^{1/(12n+1)}}<\cfrac{1}{\sqrt{2\pi}\cdot n^{n+1/2}\cdot e^{-n}}$$
and $$m!<\sqrt{2\pi}\cdot m^{m+1/2}\cdot e^{-m}\cdot e^{1/(12m)}=(2\pi)^{1/2}\cdot(\frac{n}{r})^{m+1/2}\cdot e^{-n/r}\cdot e^{1/(12m)}.$$
Then 
$$\frac{|\chi^{\lambda}(\sigma)|}{(\chi^{\lambda}(1))^{1/r}}<\frac{(2\pi)^{1/2}\cdot{(\frac{n}{r})^{m+1/2}}\cdot{r^m}\cdot{e^{-n/r}}\cdot{e^{1/(12m)}}}{(2\pi)^{1/2r}\cdot{n^{n/r+1/{2r}}}\cdot{e^{-n/r}}}=
(2\pi)^{\frac{r-1}{2r}}r^{-1/2}n^{\frac{r-1}{2r}}e^{1/12m},$$
as required.
\end{proof}

\begin{lemma}$\cite[2.3.17]{James}$\label{l:characterOfcycle} Suppose that $\lambda$ is a partition of $n$, $\chi^\lambda$ is the character of $[\lambda]$, and $\sigma$ is a permutation of shape $[n^1]$, i.e. a length $n$ cycle.
	Then \\$\chi^\lambda(\sigma)=\begin{cases} (-1)^r & \mbox{if } \lambda=(n-r, 1^r)\text{ with } 0\leq r\leq n-1;  \\ 0 & \mbox{ otherwise.} \end{cases}$
\end{lemma}

We need the following lemma to evaluate the characters of hook shape partitions on the permutations of shape $[r^m]$.
\begin{lemma}\label{l:hook_diag}
Suppose that $n$ is an odd integer and $n=r(k+l+1)$, where $k\geq 0$, $l\geq0$, and $r$ are integers.
Consider an integer $a$ satisfying $1\leq a\leq r$. Denote by $\lambda$ the partition $(a+rk, 1^{r-a+rl})$ of $n$, so that
$T^{\lambda}$ coincides with the hook $H^{\lambda}_{1\,1}$. If $\sigma\in S_n$ is of
shape $[r^{k+l+1}]$ then $\chi^\lambda(\sigma)=(-1)^{(r-a)}{k+l\choose k}$, where ${k+l\choose k}=\frac{(k+l)!}{k!\cdot l!}$ is a binomial coefficient.	
\end{lemma}
\begin{proof}
Induct on $k+l$. Denote by $\lambda(k,l)$ the partition $(a+rk, 1^{r-a+rl})$ of $n$.
If $k+l=0$ then $\chi^\lambda(\sigma)=(-1)^{(r-a)}$ by Lemma~\ref{l:characterOfcycle}.
	
Consider a permutation $\sigma'\in S_{n-r}$ of shape $[r^{k+l}]$. Assume now that $k=1$ and $l=0$. Applying the Murnaghan-Nakayama formula, we find that
$\chi^{\lambda(1,0)}(\sigma)=\chi^{\lambda(0,0)}(\sigma')=(-1)^{(r-a)}=(-1)^{(r-a)}{1\choose0}$.
The case $k=0$, $l=1$ is similar. Assume now that $k,l\geq1$. 
The Murnaghan--Nakayama formula implies that
\begin{multline*}
\chi^{\lambda(k,l)}(\sigma)=\chi^{\lambda(k-1,l)}(\sigma')+\chi^{\lambda(k,l-1)}(\sigma')=(-1)^{(r-a)}({k+l-1\choose k-1}+{k+l-1\choose l-1})=\\=(-1)^{(r-a)}({k+l-1\choose k-1}+{k+l-1\choose k})=(-1)^{(r-a)}({k+l\choose k}),
\end{multline*}
as required.
\end{proof}


\begin{lemma}$\cite[Result~3]{Rasala},\cite[Lemma~2.1]{Tong}$\label{l:MinDegree} 
Suppose that $\lambda\vdash{n}$.

$(a)$ If $n\geq15$ then the first six nontrivial minimal character degrees of $S_n$ are:
\begin{enumerate}
	\item $d_1(S_n) = n-1$ and $\lambda\in\{(n-1,1),(2,1^{n-2})\}$; 
	\item $d_2(S_n) = \frac{1}{2}n(n-3)$ and $\lambda\in\{(n-2,2),(2^2,1^{n-4})\}$; 
	\item $d_3(S_n) = d_2(S_n) + 1 = \frac{1}{2}(n-1)(n-2)$ and $\lambda\in\{(n-2,1^2),(3,1^{n-3})\}$; 
	\item $d_4(S_n) = \frac{1}{6}n(n-1)(n-5)$ and $\lambda\in\{(n-3,3),(2^3,1^{n-6})\}$; 
	\item $d_5(S_n) = \frac{1}{6}(n-1)(n-2)(n-3)$ and $\lambda\in\{(n-3,1^3),(4,1^{n-4})\}$; 
	\item $d_6(S_n) = \frac{1}{3}n(n-2)(n-4)$ and $\lambda\in\{(n-3,2,1),(3,2,1^{n-5})\}$;
\end{enumerate}
$(b)$ If $n\geq22$ then the next five smallest character degrees are: 
\begin{enumerate}
\item $d_7(S_n)=n(n-1)(n-2)(n-7)/24$ and $\lambda\in\{(n-4,4),(2^4,1^{n-8})\}$; 
\item $d_8(S_n)=(n-1)(n-2)(n-3)(n-4)/24$ and  $\lambda\in\{(n-4,1^4),(5,1^{n-5)}\}$; 
\item $d_9(S_n)=n(n-1)(n-4)(n-5)/12$ and $\lambda\in\{(n-4,2^2),(3^2,1^{n-6})\}$; 
\item  $d_{10}(S_n)=n(n-1)(n-3)(n-6)/8$ and $\lambda\in\{(n-4,3,1),(3,2^2,1^{n-7})\}$; 
\item $d_{11}(S_n)=n(n-2)(n-3)(n-5)/8$ and $\lambda\in\{(n-4,2,1^2),(4,2,1^{n-6})\}.$
\end{enumerate}
\end{lemma}

For the representations in Lemma~\ref{l:MinDegree}(a) we need explicit values of their characters.

\begin{lemma}\label{l:CharOfSmallDegree} Suppose that $n$ is an integer and take $\sigma\in S_n$.
Assume that $\sigma$ has $i_1$ cycles of length 1, $i_2$ cycles of length 2 and 
$i_3$ cycles of length 3.
Then the following holds.
\begin{enumerate}[(i)]
	\item $\chi^{(n-1,1)}(\sigma)=i_1-1$ if $n\geq 2$;
	\item $\chi^{(n-2,2)}(\sigma)=\frac{1}{2}(i_1-1)(i_1-2)+i_2-1$ if $n\geq 4$;
	\item $\chi^{(n-2,1,1)}(\sigma)=\frac{1}{2}(i_1-1)(i_1-2)-i_2$ if $n\geq 3$;
	\item $\chi^{(n-3,3)}(\sigma)=\frac{1}{6}i_1(i_1-1)(i_1-5)+i_2(i_1-1)+i_3$ if $n\geq 6$;
	\item $\chi^{(n-3,1^3)}(\sigma)=\frac{1}{6}(i_1-1)(i_1-2)(i_1-3)-i_2(i_1-1)+i_3$ if $n\geq 4$;
	\item $\chi^{(n-3,2,1)}(\sigma)=\frac{1}{3}i_1(i_1-2)(i_1-4)-i_3$ if $n\geq 5$.
\end{enumerate}

\end{lemma}
\begin{proof} For claims $(i)-(iii)$, see \cite[Exercise 4.15]{FulHar}.

Now we consider the remaining cases: $$\lambda\in\{(n-3,3),(n-3,1^3),(n-3,2,1)\}.$$
We often use further that the coefficient of $x_1^{i_1}x_2^{i_2}\ldots x_k^{i_k}$ in the expansion of $(x_1+x_2\ldots+x_k)^m$
equals the multinomial coefficient ${m\choose i_1,i_2,\ldots,i_k}=\frac{m!}{i_1!\cdot i_2!\ldots\cdot i_k!}$.

{\bf Case $\lambda=(n-3,3)$}. Proposition~\ref{p:frob_formula} yields 
$$\chi^\lambda(\sigma)=[(x_1-x_2)(x_1+x_2)^{i_1}(x_1^2+x_2^2)^{i_2}(x_1^3+x_2^3)^{i_3}x_1^{n-i_1-2i_2-3i_3}]_{(n-2,3)},$$ i.e. the coefficient of $x_1^{n-2}x_2^3$ in this product.
Four factors involve $x_2$, namely $(x_1-x_2)$, $(x_1+x_2)^{i_1}$, $(x_1^2+x_2^2)^{i_2}$, and $(x_1^3+x_2^3)^{i_3}$. The sum of powers of $x_2$ must be equal to 3.
It is easy to see that there are five possibilities: $3=0+0+0+3$,
$3=0+1+2+0$, $3=1+0+2+0$, $3=0+3+0+0$, $3=1+2+0+0$. The corresponding factors containing $x_2$ are 
$i_3\cdot x_1^{3i_3-3}x_2^3$,
$i_1i_2\cdot x_1^{i_1-1}x_2x_1^{2i_2-2}x_2^2$, $-i_2\cdot x_2x_1^{2i_2-2}x_2^2$, ${i_1\choose 3}\cdot x_1^{3i_1-3}x_2^3$,
$-{i_1\choose 2}\cdot x_2x_1^{i_1-2}x_2^2$. Therefore, the coefficient of $x_1^{n-2}x_2^3$ equals 
$i_3+i_1i_2-i_2+\frac{i_1(i_1-1)(i_1-2)}{6}-\frac{i_1(i_1-1)}{2}=i_3+i_1i_2-i_2+\frac{i_1(i_1-1)(i_1-2)-3i_1(i_1-1)}{6}=
\frac{i_1(i_1-1)(i_1-5)}{6}+i_3+i_2(i_1-1)$, as claimed.

{\bf Case $\lambda=(n-3,1^3)$}. Proposition~\ref{p:frob_formula} yields
$$\chi^\lambda(\sigma)=[\Delta(x)\cdot(x_1+x_2+x_3+x_4)^{i_1}(x_1^2+x_2^2+x_3^2+x_4^2)^{i_2}(x_1^3+x_2^3+x_3^3+x_4^3)^{i_3}]_{(n,3,2,1)},$$
and we need to find the coefficient of $x_1^nx_2^3x_3^2x_4$ in this product.
In particular, if $i_1=i_2=i_3=0$ then $\chi^\lambda(\sigma)=0$ and this is consistent with the claimed formula.
So we can assume that at least one of the numbers $i_1$, $i_2$, and $i_3$ is nonzero.

Furthermore, we continue our calculations using the Murnaghan--Nakayama formula rather than the Frobenius formula because the latter in this case requires the consideration of a large number of cases.

Rim hooks are removed in the same way for $n\geq 7$, so first we suppose that $4\leq n\leq 6$.
If $n=4$ then $\lambda=(1^4)$ corresponds to the alternating representation. If $n=5$ then $\lambda'=(4,1)$
corresponds to the standard representation. If $n=6$ then $\lambda'=(4,1^2)$ is the representaion of claim~$(iii)$ of the lemma. In particular, we know the formulas for characters in these cases.
This is a routine check that in all cases their values coincide with $\frac{1}{6}(i_1-1)(i_1-2)(i_1-3)-i_2(i_1-1)+i_3$.

Assume now that $n\geq7$. We consider three cases: for $i_1\neq0$, $i_2\neq0$, and $i_3\neq0$. If $i_1\neq0$ then we can assume that $\sigma$ fixes $n$ and denote by $\sigma_1$ the permutation
in $S_{n-1}$ acting on $\{1,2,\ldots, n-1\}$ as $\sigma$. Observe that only the hooks $H^\lambda_{1\,n-3}$ and $H^\lambda_{4\,1}$ of $T^\lambda$ are of length one, so Proposition~\ref{NM-formula} implies that $\chi^\lambda(\sigma)=\chi^{(n-4,1^3)}(\sigma_1)+\chi^{(n-3,1^2)}(\sigma_1)$.
By induction, we can assume that $\chi^{(n-2,1^3)}(\sigma_1)=\frac{1}{6}(i_1-2)(i_1-3)(i_1-4)-i_2(i_1-2)+i_3$,
and by $(ii)$ we know that  $\chi^{(n-3,1^2)}(\sigma_1)=\frac{1}{2}(i_1-2)(i_1-3)-i_2$.
Therefore, 
\begin{multline*}
\chi^\lambda(\sigma)=\frac{1}{6}(i_1-2)(i_1-3)(i_1-4)-i_2(i_1-2)+i_3+\frac{1}{2}(i_1-2)(i_1-3)-i_2=\\=
\frac{1}{6}(i_1-2)(i_1-3)(i_1-4+3)-i_2(i_1-1)+i_3=\frac{1}{6}(i_1-1)(i_1-2)(i_1-3)-i_2(i_1-1)+i_3,
\end{multline*}
as required.

Suppose that $i_2\neq0$ and $\sigma$ switches the symbols $n-1$ and $n$. Denote by $\sigma_1$ the permutation
in $S_{n-2}$ acting on $\{1,2,\ldots, n-2\}$ as $\sigma$. Observe that only the hooks $H^\lambda_{1\,n-4}$ and $H_{3\,1}$ of $T^\lambda$ are of length two, so Proposition~\ref{NM-formula} implies that
$\chi^\lambda(\sigma)=\chi^{(n-5,1^3)}(\sigma_1)-\chi^{(n-3,1)}(\sigma_1)$. By induction and $(i)$,
we find that
\begin{multline*}
\chi^\lambda(\sigma)=\frac{1}{6}(i_1-1)(i_1-2)(i_1-3)-(i_2-1)(i_1-1)+i_3-i_1+1=\\=
\frac{1}{6}(i_1-1)(i_1-2)(i_1-3)-i_2(i_1-1)+i_3,
\end{multline*}
as required.

Finally, suppose that $i_3\neq0$ and $\sigma$ includes the cycle $(n-2,n-1,n)$ in its cycle decomposition.
Denote by $\sigma_1$ the permutation
in $S_{n-3}$ acting on $\{1,2,\ldots, n-3\}$ as $\sigma$. Observe that only the hooks $H^\lambda_{1\,n-5}$ and $H^\lambda_{2\,1}$ of $T^\lambda$ are of length three, so Proposition~\ref{NM-formula} implies that
$\chi^\lambda(\sigma)=\chi^{(n-6,1^3)}(\sigma_1)+\chi^{(n-3)}(\sigma_1)$. 
Clearly, $\chi^{(n-3)}(\sigma_1)=1$. By induction,
we find that
\begin{multline*}
\chi^\lambda(\sigma)=\frac{1}{6}(i_1-1)(i_1-2)(i_1-3)-i_2(i_1-1)+i_3-1+1=\\=
\frac{1}{6}(i_1-1)(i_1-2)(i_1-3)-i_2(i_1-1)+i_3,
\end{multline*}
as required.

{\bf Case $\lambda=(n-3,2,1)$}. Proposition~\ref{p:frob_formula} yields
$$\chi^\lambda(\sigma)=[\Delta(x)(x_1+x_2+x_3)^{i_1}(x_1^2+x_2^2+x_3^2)^{i_2}(x_1^3+x_2^3+x_3^3)^{i_3}]_{(n-1,3,1)},$$
and we need to find the coefficient of $x_1^{n-1}x_2^3x_3$ in this product.
First we note that $$\Delta(x)=(x_1-x_2)(x_1-x_3)(x_2-x_3)=x_1^2x_2-x_1x_2^2+x_2^2x_3-x_2x_3^2+x_1x_3^2-x_1^2x_3.$$
Since $x_3^2$ does not divide $x_1^{n-1}x_2^3x_3$, we can use only 
$x_1^2x_2-x_1x_2^2+x_2^2x_3-x_1^2x_3$ to obtain the term $x_1^{n-1}x_2^3x_3$ in the product.

Assume that we avoid  $x_1^{3i_3}$ in the factor $(x_1^3+x_2^3+x_3^3)^{i_3}$.
Then we can take only the term $i_3x_1^{3(i_3-1)}x_2^3$. Note that $\Delta(x)$ contains the factor $(x_2-x_3)$ and since we chose $x_2^3$, we use $-x_3$ in this factor. Now from all the other factors in the product we can use only the powers of $x_1$; hence, in this case we get $-i_3x_1^{n}x_2^3x_3$.
So we can assume now that we use $x_1^{3i_3}$ in $(x_1^3+x_2^3+x_3^3)^{i_3}$.

Assume that we avoid $x_1^{2i_2}$ in the factor $(x_1^2+x_2^2+x_3^2)^{i_2}$.
Then we can use only the term $i_2x_1^{2(i_2-1)}x_2^2$. Now we need to find the factors where we take $x_2$ and $x_3$ to obtain the product $x_2x_3$. Observe that in $\Delta(x)$ we can use either $x_1^2x_2$ or $-x_1^2x_3$.
In the first case we must take $i_1x_1^{i_1-1}x_3$ in $(x_1+x_2+x_3)^{i_1}$ and in the second 
case $i_1x_1^{i_1-1}x_2$. Therefore, in the product for $x_1^{2i_2}$ we get $i_2(i_1-i_1)\cdot(-i_3x_1^{n}x_2^3x_3)=0$.
So we can assume now that we use $x_1^{2i_2}$ in $(x_1^2+x_2^2+x_3^2)^{i_2}$.

It remains to find the coefficient of $x_1^{n-2i_2-3i_3}x_2^3x_3$ in the product $(x_1^2x_2-x_1x_2^2+x_2^2x_3-x_1^2x_3)(x_1+x_2+x_3)^{i_1}$.
Once we fix a term in the first factor, the term in $(x_1+x_2+x_3)^{i_1}$ is determined uniquely.
Namely, for $x_1^2x_2$ we take  $\frac{i_1(i_1-1)(i-2)}{2}x_1^{i_1-3}x_2^2x_3$, for 
$-x_1x_2^2$ we take  $i_1(i_1-1)x_1^{i_1-2}x_2x_3$, for $x_2^2x_3$ we take $i_1x_1^{i_1-1}x_2$,
and for $-x_1^2x_3$ we take $\frac{i_1(i_1-1)(i_1-2)}{6}x_1^{i_1-3}x_2^3$.
Thus, in this case we obtain
\begin{multline*}
\Big(\frac{i_1(i_1-1)(i_1-2)}{3}-i_1(i_1-1)+i_1\Big)x_1^{n-1}x_2^2x_3=\\=\frac{i_1}{3}(i_1^2-3i_1+2-3i_1+3+3)x_1^{n-1}x_2^2x_3=\frac{i_1(i_1-2)(i_1-4)}{3}x_1^{n-1}x_2^3x_3.
\end{multline*}

So the final coefficient of $x_1^{n-1}x_2^3x_3$ is $\frac{i_1(i_1-2)(i_1-4)}{3}-i_3$, as claimed.

\end{proof}

The following result helps us to find the dimensions of eigenspaces of the elements using the character values.
\begin{lemma}\label{l:calc_centraliser} Consider a finite group $G$ and an ordinary 
representation of $\rho:G\rightarrow GL(V)$. Take $g\in G$ with $|g|=n$ and $\eta\in\mathbb{C}$ with $\eta^n=1$.
If $\chi$ is the character of $\rho$ then
$$n\cdot\dim (V^\eta(\rho(g)))=\sum_{i=0}^{i<n}\chi(g^i)\overline{\eta}^i,$$
where $V^\eta(\rho(g))=\operatorname{ker}(\rho(g)-\eta\cdot\operatorname{Id_V})$, i.e. the eigenspace
of $\rho(g)$ associated to $\eta$.
\end{lemma}
\textsf{}\begin{proof} Denote by $\rho'$ the restriction of $\rho$ to $H=\langle g \rangle$ and by
$\chi'$ the character of $\rho'$. Consider an irreducible representation of $H$ which maps $g$ to $\eta g$ and denote by $\psi$ its character.
Then the inner product $$\langle \chi', \psi\rangle=\frac{1}{|H|}\sum_{i=0}^{i<n}\chi'(g^i)\overline{\psi(g^i)}=\frac{1}{n}\sum_{i=0}^{i<n}\chi(g^i)\overline{\eta}^i$$
equals the number of irreducible constituents of $\chi'$ with character $\psi$. Clearly, the direct sum of these constituents is the eigenspace of $\rho'(g)$ associated to $\eta$.
Therefore, $\sum_{i=0}^{i<n}\chi(g^i)\overline{\eta}^i=n\cdot\dim(V^\eta(\rho(g)))$, as claimed.
\end{proof}

\begin{lemma}\label{l:inequality}
Suppose that $n\geq23$. Then the following inequalities hold.
\begin{enumerate}[(i)]
	\item $\sqrt{n}(n-2)(n-7)>54\cdot(n-1)$;
	\item $\sqrt{n}(n-2)^2(n-7)^2>8\cdot(24)^2\cdot(n-1)$.
\end{enumerate}
\end{lemma}
\begin{proof}
	To prove $(i)$, we show that $\sqrt{n}(n-2)(n-7)>54n$. This inequality is equivalent to 
	$\frac{n-2}{\sqrt{n}}(n-7)>54$. Since $\sqrt{n}>4$,
	we have $\frac{n-2}{\sqrt{n}}=\sqrt{n}-\cfrac{2}{\sqrt{n}}>4-1/2=7/2$.
	Then $\cfrac{n-2}{\sqrt{n}}(n-7)>\frac{7}{2}16=56$. So 
	$\sqrt{n}(n-2)(n-7)>54n>54(n-1)$.
	
	Now we prove $(ii)$. By $(i)$, we have $\sqrt{n}(n-2)^2(n-7)^2\geq54(n-2)(n-7)(n-1)$.
	So it suffices to prove that $(n-2)(n-7)\geq\frac{8(24)^2}{54}$.
	This inequality is true for $n=23$ and hence it is true for all $n\geq23$.

\end{proof}

\section{Proof of Theorem~\ref{th:1}}

In this section we prove Theorem~\ref{th:1}. Throughout,
we suppose that $n\geq3$ and $\sigma$ is a permutation of shape $[r^m1^{n-rm}]$
in $S_n$, where $r\geq2$. We fix a~partition $\lambda=(\lambda_1,\lambda_2,\ldots,\lambda_k )$ of $n$
and an irreducible $S_n$-module $V$ corresponding to $\lambda$.
 The proof is split into several lemmas.

\begin{lemma}\label{l:special_reps}
If $\lambda\in\{(1^n),(n-1,1),(2,1^{n-2})\}$ then Theorem~\ref{th:1} holds.
\end{lemma}
\begin{proof}	
Assume first that $\lambda$ gives the alternating representation, i.e. $\lambda=(1^n)$ and $V$ is one-dimensional.
Then we have $\chi^{\lambda}(\sigma)=\operatorname{sgn}(\sigma)$. If $\sigma$ is even then
$|\sigma|\cdot\dim C_V(\rho(\sigma))= \sum_{i=0}^{i<|\sigma|}\chi^\lambda(\sigma^i)=\sum_{i=0}^{i<|\sigma|}1=|\sigma|$.
Hence, $\sigma$ acts trivially on $V$ and the minimal polynomial equals $x-1$. If $\sigma$ is odd then $|\sigma|=r$ is even. Therefore, 
$|\sigma|\cdot\dim{V^{-1}}(\rho(\sigma))=\sum_{i=0}^{i<|\sigma|}\chi^\lambda(\sigma^i)(-1)^i=|\sigma|$. 
In this case $V=V^{-1}(\rho(\sigma))$ and the minimal polynomial equals $x+1$. 

Now take $\lambda=(n-1,1)$, i.e. $\lambda$ corresponds to the standard representation. 
Denote by $i_1=n-rm$ the number of length one cycles in the cycle decomposition of~$\sigma$. Lemma~\ref{l:CharOfSmallDegree}
implies that $\chi^\lambda(\sigma^0)=n-1$ and $\chi^\lambda(\sigma^j)=i_1-1$
for $1\leq j\leq r-1$.
If $\eta$ is a root of the polynomial $x^r-1$ then $r\cdot\dim V^{\eta}(\rho(\sigma))=n-1+(i_1-1)(\eta+\eta^2+\ldots+\eta^{r-1})=
n-i_1+(i_1-1)(1+\eta+\ldots+\eta^{r-1})$.
Therefore, $$r\cdot\dim V^{\eta}(\rho(\sigma))=\begin{cases} 
n-i_1+(i_1-1)r & \mbox{ for }\eta=1,\\
n-i_1 &  \mbox{ for }\eta\neq1.
 \end{cases}$$

Thus, $r\cdot\dim V^{\eta}(\rho(\sigma))=0$ if and only if $i_1=0$, $\eta=1$, and $r=n$. Consequently, if $\sigma$ is not of shape $[n^1]$ then the minimal polynomial
equals $x^r-1$, and for $\sigma=[n^1]$ the minimal polynomial equals $\frac{x^n-1}{x-1}$.

Finally, take $\lambda=(2,1^{n-2})$. 
If $\sigma$ is even then $\chi^{\lambda}(\sigma)=\chi^{(n-1,1)}(\sigma)$
and this case is similar to the previous one. 
If $\sigma$ is odd then $r$ is even. If $\eta$ is a root of the polynomial $x^r-1$ then 
$$r\cdot\dim V^\eta(\rho(\sigma))=n-1+(i_1-1)(-\eta+\eta^2-\ldots-\eta^{r-1})=
n-i_1+(i_1-1)(1-\eta+\ldots-\eta^{r-1}).$$ 
Therefore,
$$r\cdot\dim V^{\eta}(\rho(\sigma))=\begin{cases} 
n-i_1+(i_1-1)r & \mbox{ for }\eta=-1,\\
n-i_1 &  \mbox{ for }\eta\neq-1.
\end{cases}$$
Thus, $r\cdot\dim V^{\eta}(\rho(\sigma))=0$ if and only if $i_1=0$, $\eta=-1$,
and $r=n$. 
\end{proof}

\begin{lemma}\label{l:n_less_23} Theorem~\ref{th:1} holds if $n\leq 22$ and $\sigma$ is of shape $[r^m]$, where $rm=n$.
\end{lemma}
\begin{proof} By Lemma~\ref{l:special_reps}, we can assume that $\lambda\not\in\{(1^n), (n-1,1), (2,1^{n-2})\}$.
We need to prove that the minimal polynomial of $\sigma$ 
is $x^r-1$. Since $\sigma$ is a power of a $n$-cycle,
we can assume that $\sigma=[n^1]$. Now we apply GAP to establish the assertion.
If $\eta$ is a root of $x^n-1$ and $V$ is an irreducible representation of $S_n$ 
then we find dim $V^\eta(\rho(\sigma))$ with the aid of Lemma~\ref{l:calc_centraliser}.
\end{proof}

\begin{lemma}\label{l:case_small_reps} Let $n\geq7$. Assume that either $\lambda$ or $\lambda'$
belongs to $$M=\{(n-2,2),(n-2,1,1),(n-3,3),(n-3,1^3),(n-3,2,1)\}.$$
If $\sigma$ is of shape $[r^m]$, where $n=rm$ and $r\geq2$, then $\chi^\lambda(1)>\sum_{i=1}^{i<r}|\chi^\lambda(\sigma^i)|$.
\end{lemma}
\begin{proof}
Since $\chi^{\lambda'}(\rho)=\operatorname{sgn}(\rho)\cdot\chi^{\lambda}(\rho)$,
we can assume that $\lambda\in M$. Now we consider each possibility for $\lambda$.

Suppose that $\lambda=(n-2,2)$. Note that $\sigma^i$ is of shape $[k^t]$, where $2\leq k\leq r$ and $kt=n$. Lemma~\ref{l:CharOfSmallDegree}
implies that $\chi^{\lambda}(\sigma^i)=n/2$ if $k=2$ and $\chi^{\lambda}(\sigma^i)=0$ if $k>2$. Observe that $k=2$ for at most one value of $i$. Therefore, $\sum_{i=1}^{i<r}|\chi^\lambda(\sigma^i)|\leq n/2$.
On the other hand, Lemma~\ref{l:MinDegree} implies that
$\chi^{\lambda}(1)=~\frac{n(n-3)}{2}$ and hence $\chi^{\lambda}(1)>n/2$ for $n\geq5$. So $\chi^\lambda(1)>\sum_{i=1}^{i<r}|\chi^\lambda(\sigma^i)|$.

Suppose that $\lambda=(n-2,1,1)$. Lemma~\ref{l:CharOfSmallDegree} implies 
that $\chi(\sigma^i)=1-n/2=(2-n)/2$ if $\sigma^i$ is of shape $[2^{n/2}]$ and $\chi(\sigma^i)=1$ otherwise. Again, the permutation $\sigma^i$ is of shape [$2^{n/2}$] only for at most one value of $i$. 
On the other hand, we know that $\chi^\lambda(1)=(n-1)(n-2)/2$. Therefore,
if $n\geq6$ then $\chi^\lambda(1)>n-1+(n-2)/2\geq\sum_{i=1}^{i<r}|\chi^\lambda(\sigma^i)|$.

Suppose that $\lambda=(n-3,3)$. By Lemma~\ref{l:CharOfSmallDegree}, if $1\leq i\leq r - 1$ then
$$\chi^{\lambda}(\sigma^i)=
\begin{cases}
\chi^\lambda(\sigma^i)=-n/2 & \mbox{ if } \sigma^i\mbox{ is of shape }[2^{n/2}], \\
\chi^\lambda(\sigma^i)=n/3 & \mbox{ if } \sigma^i\mbox{ is of shape }[3^{n/3}], \\
0 & \mbox{ otherwise.}
\end{cases}$$
Observe that $\sigma^i$ is of shape $[2^{n/2}]$ for at most one value of $i$
and $\sigma^i$ is of shape $[3^{n/3}]$ for at most two values of $i$.
Therefore, we have $\sum_{i=1}^{i<r}|\chi^\lambda(\sigma^i)|\leq n/2+2n/3=7n/6$.
On the other hand, $\chi^\lambda(1)=\frac{1}{6}n(n-1)(n-5)$.
Consequently, if $n\geq7$ then $(n-1)(n-5)>7$ and hence $\chi^\lambda(1)>\sum_{i=1}^{i<r}|\chi^\lambda(\sigma^i)|$.

Suppose that $\lambda=(n-3,1^3)$. By Lemma~\ref{l:CharOfSmallDegree}, if $1\leq i\leq r - 1$ then
$$\chi^{\lambda}(\sigma^i)=
\begin{cases}
\chi^\lambda(\sigma^i)=n/2-1 & \mbox{ if } \sigma^i\mbox{ is of shape }[2^{n/2}], \\
\chi^\lambda(\sigma^i)=n/3-1 & \mbox{ if } \sigma^i\mbox{ is of shape }[3^{n/3}], \\
-1 & \mbox{ otherwise.}
\end{cases}$$
Therefore, we have $\sum_{i=1}^{i<r}|\chi^\lambda(\sigma^i)|\leq n/2+2n/3+n-3=13n/6-3$.
On the other hand, $\chi^\lambda(1)=\frac{1}{6}(n-1)(n-2)(n-3)$.
Since $n\geq7$, we have $\frac{1}{6}(n-1)(n-2)(n-3)\geq 20(n-1)/6>13n/6$ and hence $\chi^\lambda(1)>\sum_{i=1}^{i<r}|\chi^\lambda(\sigma^i)|$.

Finally, suppose that $\lambda=(n-3,2,1)$. By Lemma~\ref{l:CharOfSmallDegree}, 
if $\chi^\lambda(\sigma^i)\neq0$ then
$\chi^\lambda(\sigma^i)=-n/3$ and $\sigma^{i}$ is of shape $[3^{n/3}]$.
Therefore, $\sum_{i=1}^{i<r}|\chi^\lambda(\sigma^i)|\leq2n/3$.
On the other hand, $\chi^\lambda(1)=\frac{1}{3}n(n-2)(n-4)$ and hence
$\chi^\lambda(1)>\sum_{i=1}^{i<r}|\chi^\lambda(\sigma^i)|$ for $n\geq5$.
\end{proof}

\begin{lemma}\label{l:rm_equals_n}
Suppose that $n\geq23$ and $\sigma$ is of shape $[r^m]$, where $rm=n$ and $r\geq2$.
If $\chi$ is an irreducible character of $S_n$ and $\chi(1)>n-1$
then $\chi(1)>\sum_{i=1}^{i<r}|\chi(\sigma^i)|$.
\end{lemma}
\begin{proof}
Assume that $\chi(1)<{n(n-1)(n-2)(n-7)/24}$.
Lemma~\ref{l:MinDegree} shows that
$\lambda$ or $\lambda'$ belongs to $\{(n-2,2),(n-2,1,1),(n-3,3),(n-3,1^3),(n-3,2,1)\}$
and then $\chi(1)>\sum_{i=1}^{i<r}|\chi(\sigma^i)|$ by Lemma~\ref{l:case_small_reps}.

Assume now that $\chi(1)\geq{n(n-1)(n-2)(n-7)/24}$.

Clearly, if $1\leq i<r$ then $\sigma^i$ is of shape $[k^t]$ for some integer $k$ and $t$ with $2\leq k\leq r$ and $kt=n$.
Lemma~\ref{l:ch_estimate} implies that
$$\frac{|\chi^{\lambda}(\sigma^i)|}{(\chi^{\lambda}(1))^{1/k}}<(2\pi)^{\frac{k-1}{2k}}k^{-1/2}n^{\frac{k-1}{2k}}e^{1/12t}.$$

Assume that $k=2$. Then 
$(2\pi)^{\frac{k-1}{2k}}k^{-1/2}n^{\frac{k-1}{2k}}e^{1/12t}=(\frac{\pi}{2})^{\frac{1}{4}}e^{1/6n}n^\frac{1}{4}<\frac{3}{2}n^\frac{1}{4}$.
Therefore, in this case 
$$|\chi(\sigma^i)|<\frac{3}{2}n^\frac{1}{4}\sqrt{\chi(1)}$$

Assume now that $k\geq3$. Then we have $(2\pi)^{\frac{k-1}{2k}}\leq\sqrt{2\pi}$, $k^{-1/2}\leq\cfrac{1}{\sqrt{2}}$, $n^\frac{k-1}{2k}\leq\sqrt{n}$,
and $e^{1/12t}\leq1.1$. Moreover, $\chi(1)^{1/k}\leq\chi(1)^\frac{1}{3}$ and hence
$$|\chi(\sigma^i)|\leq 1.1\sqrt{\pi}\sqrt{n}\chi(1))^\frac{1}{3}\leq 2\sqrt{n}\chi^\lambda(1)^\frac{1}{3}.$$

Now we prove that $\chi(\sigma^i)<\chi(1)/(n-1)$.
For $k=2$, it suffices to show that $\frac{3}{2}n^\frac{1}{4}\sqrt{\chi(1)}<\chi(1)/(n-1)$. This is equivalent to
$\frac{9}{4}\sqrt{n}(n-1)^2<\chi(1)$.
If $\chi(1)\geq n(n-1)(n-2)(n-7)/24$ then Lemma~\ref{l:inequality} implies that
$\chi(1)\geq \sqrt{n}(n-1)\cdot54/24(n-1)=9/4\sqrt{n}(n-1)^2$, as required.

For $k\geq3$, it suffices to show that $2n^\frac{1}{2}\chi(1)^\frac{1}{3}<\chi(1)/(n-1)$. This is equivalent to
$8n\sqrt{n}(n-1)^3<\chi(1)^2$.
We verify that $$(n(n-1)(n-2)(n-7)/24)^2\geq 8n\sqrt{n}(n-1)^3.$$
This is equivalent to
$\sqrt{n}(n-2)^2(n-7)^2\geq 8(24)^2(n-1)$, which is true by Lemma~\ref{l:inequality}. Therefore,
we infer that $(\chi(1))^2>(n(n-1)(n-2)(n-7)/24)^2\geq 8n\sqrt{n}(n-1)^3$, as required.	
\end{proof}

\begin{lemma}\label{l:general_case} Theorem~\ref{th:1} holds if $\lambda\not\in\{(1^n,(n-1,1),(2,1^{n-2})\}$.
\end{lemma}
\begin{proof}
Proceed by induction on $n$. Assume that $\sigma$ is of shape $[r^m1^{n-rm}]$.
Suppose that $rm<n$ and $n\neq5,7$. Then we may assume that $\sigma$ belongs to the stabilizer $S_{n-1}$ of the point $n$. The Branching Law yields $[\lambda]\downarrow S_{n-1}=\sum_{i:\lambda_i>\lambda_{i+1}}[\lambda^{i-}]$.
Take some partition $\mu$ of $n-1$ occurring in the latter sum. 
If $\mu\not\in\{(1^{n-1}),(n-2,1),(2,1^{n-2})\}$ then by the induction hypothesis 
 $\sigma$ has an eigenvector for every root $\eta$ of the polynomial
$x^r-1$ and hence the minimal polynomial equals $x^r-1$.
If $\mu=(1^{n-1})$ then $\lambda$ is either $(1^n)$ or $(2,1^{n-2})$; a contradiction.
If $\mu=(n-2,1)$ then, by assumption, $\lambda$ is either $(n-2,2)$ or $(n-2,1,1)$.
In the first case $[\lambda]\downarrow S_{n-1}$ has a summand for the partition $(n-3,2)$ and
in the second case for $(n-3,1,1)$.
Using the induction hypothesis for these summands,
we infer that the minimal polynomial is $x^r-1$. If $n=5$ or $7$ then we apply the same argument avoiding the exceptional cases of the theorem.

Therefore, we can assume that $n=rm$. By Lemma~\ref{l:n_less_23}, it remains to consider the cases $n\geq23$.
If $\eta$ is a root of the polynomial $x^r-1$ then $$n\cdot\dim V^\eta(\rho(\sigma))=\sum_{i=0}^{i<r}\chi^\lambda(\sigma^i)\overline{\eta}^i.$$
Since the sum gives a real number and $\chi^\lambda(\sigma^i)$
are integers, we have $$n\cdot \dim V^\eta(\rho(\sigma))=\sum_{i=0}^{i<r}\chi^\lambda(\sigma^i)\operatorname{Re}(\overline{\eta}^i)\geq \chi^\lambda(1)-\sum_{i=1}^{i<r}|\chi^\lambda(\sigma^i)|.$$
Now Lemma~\ref{l:rm_equals_n} implies that $\dim V^\eta(\rho(\sigma))>0$ and hence
the minimal polynomial of $\sigma$ equals $x^r-1$, as required.

\end{proof}

\section{Proof of Corollary~\ref{cor:1}}
In this section we prove Corollary~\ref{cor:1}.
Throughout, we assume that $(\rho,V)$ is an irreducible representation of $A_n$, where $n\geq5$,
and $\chi$ is the character of $\rho$. By $\sigma$ we denote a permutation of $A_n$
of shape $[r^m1^{n-rm}]$, where $r\geq2$ and $rm\leq n$.  
The minimal polynomial of $\rho(\sigma)$ on $V$ is denoted by $\mu_{\rho(\sigma)}(x)$.

\begin{lemma}\label{l:A_n_standard} If $\rho$ is the standard representation
then $\mu_{\rho(\sigma)}(x)\neq x^r-1$ if and only if
$r=n$, where $n$ is odd.
\end{lemma}
\begin{proof}
Take the partition $\lambda=(n-1,1)$ of $n$
corresponding to the standard representation of $S_n$.
Since $n\geq5$, we have $\lambda\neq\lambda'$ and hence
$[\lambda]\downarrow A_n$ is the standard representation of $A_n$.
For every $\tau\in A_n$ we have $\chi(\tau)=\chi^\lambda(\tau)$.
Then for every $\eta$ with $\eta^r=1$ the dimensions of eigenspaces
of $\sigma$ for $A_n$ and $S_n$ coincide. Now Theorem~\ref{th:1} implies that
$\mu_{\rho(\sigma)}(x)\neq x^r-1$ only if $\sigma$ is a length $n$ cycle .
If $\sigma$ is such a cycle then $n$ is odd and hence $\mu_{\rho(\sigma)}(x)=\frac{x^n-1}{x-1}$ by Theorem~\ref{th:1}.
\end{proof}

\begin{lemma} Corollary~\ref{cor:1} holds if $\rho$ is not the standard representation.
\end{lemma}	
\begin{proof}
Induct on $n$. If $n=5$ or $6$ then the assertion can be verified using the character tables of $A_5$ and $A_6$. Observe that there are two irreducible characters of $A_5$ that correspond to the self-associated partition $(3,1,1)$. In these cases if $\eta$ is a primitive 5th root of unity then either $\eta$ and $\overline{\eta}$ or $\eta^2$ and $\overline{\eta^2}$ are not roots of $\mu_{\rho((1,2,3,4,5))}(x)$.

Assume now that $\rho$ is obtained from the irreducible representation of $S_n$
corresponding to a partition $\lambda$ of $n$. If $\lambda\neq\lambda'$ then we can argue as in Lemma~\ref{l:A_n_standard} and the assertion follows from Theorem~\ref{th:1}.
Suppose that $\lambda=\lambda'$. If $\chi(\sigma)=\chi^\lambda(\sigma)/2$ then the assertion follows from Theorem~\ref{th:1}. Denote the lengths of hooks along the main diagonal of $T^\lambda$ by $l_1$,$l_2$,$\ldots$,$l_k$. It follows from Proposition~\ref{p:chars_of_An} that $\chi(\sigma)\neq\chi^\lambda(\sigma)/2$
only if $\sigma$ is of shape $[l_1^1l_2^1\ldots{l}_k^1]$. Since $l_i$ are always distinct,
we infer that either $\sigma=[(n-1)^11^1]$ and $n$ is even or $\sigma=[n^1]$ and $n$ is odd.

Assume that $\sigma=[(n-1)^11^1]$, that is,	 a length $n-1$ cycle.
In this case $\rho$ results from the representation of $S_n$ corresponding to the
self-associated partition $\lambda=(\frac{n}{2}, 2, 1^\frac{n-4}{2})$.
By Proposition~\ref{p:reps_of_An}, we have $[\lambda]\downarrow A_n=V_1\oplus V_2$.
Consider the subgroups $A_{n-1}$ and $S_{n-1}$ containing $\sigma$ and stabilizing the point $n$.
By the Branching Law,  
$[\lambda]\downarrow S_{n-1}=[(\frac{n-2}{2}, 2, 1^{\frac{n-4}{2}})]+[(\frac{n}{2},1^\frac{n-2}{2})]+[(\frac{n}{2}, 2, 1^{\frac{n-6}{2}})]$.
All the summands are not equivalent to the standard representation of $S_{n-1}$ and only the partition $(\frac{n}{2},1^{\frac{n-2}{2}})$
is self-associated. By Proposition~\ref{p:reps_of_An}, restricting these representations to $A_{n-1}$, yields four summands. By induction,
in each of them the minimal polynomial of $\sigma$ equals $x^{n-1}-1$.
Therefore, each irreducible constituent of the restrictions of the representations of $A_n$ to $A_{n-1}$ equals one of the four mentioned above and hence the minimal polynomials of $\sigma$ on $V_1$ and $V_2$ equal $x^{n-1}-1$.

Assume that $\sigma=[n^1]$, that is, a length $n$ cycle with odd $n$.
In this case $\rho$ results from a representation of $S_n$ corresponding to the self-associated partition 
$\lambda=(\frac{n+1}{2}, 1^\frac{n-1}{2})$.
By Proposition~\ref{p:chars_of_An}, if $\tau$ is not a~length $n$ cycle then $\chi(\tau)=\chi^\lambda(\tau)/2$.
Moreover, Lemma~\ref{l:characterOfcycle} and Proposition~\ref{p:chars_of_An} imply that $\chi(\sigma)=\frac{1}{2}\big((-1)^\frac{n-1}{2}\pm\sqrt{(-1)^\frac{n-1}{2}n}\big)$.

We prove that $|\chi(1)|\geq\sum_{i=1}^{n-1}|\chi(\sigma^i)|.$
The hook-length formula yields $\chi(1)=\chi^\lambda(1)/2=\frac{n}{2}\cdot{n-1\choose\frac{n-1}{2}}$. We show that 
$|\chi(\sigma^i)|\leq\frac{1}{2}{n-1\choose\frac{n-1}{2}}$ for each $i\in\{1,\ldots,n-1\}$.
If $\sigma^i$ is not a cycle then $\sigma^i$ is of shape $[r^m]$, where $r,m\geq2$. Now Lemma~\ref{l:hook_diag} yields $\chi^{\lambda}(\sigma^i)=(-1)^{(r-1)/2}\cdot{m-1\choose\frac{m-1}{2}}$ and hence $|\chi(\sigma^i)|=\frac{1}{2}|\chi^{\lambda}(\sigma^i)|\leq\frac{1}{2}{n-1\choose\frac{n-1}{2}}$, as claimed.
If $\sigma^i$ is a cycle then $\chi(\sigma^i)=\frac{1}{2}((-1)^{\frac{(r-1)}{2}}\pm\sqrt{(-1)^{\frac{(r-1)}{2}}n})$.
It follows that $|\chi(\sigma^i)|\leq\frac{1}{2}(1+\sqrt{n})$. Since ${n-1\choose\frac{n-1}{2}}>{n-1\choose1}=n-1$, we have ${n-1\choose\frac{n-1}{2}}\geq n\geq2\sqrt{n}-1>1+\sqrt{n}$, and hence $|\chi(\sigma^i)|<\frac{1}{2}{n-1\choose\frac{n-1}{2}}$, as claimed.

Now consider a complex number $\eta$ such that $\eta^n=1$. Since $$\operatorname{Re}(\chi(\sigma^i)\overline{\eta}^i)\geq-|\chi(\sigma^i)\overline{\eta}^i|=-|\chi(\sigma^i)|,$$
we have
$$n\cdot\dim (V^\eta(\rho(\sigma)))=\sum_{i=0}^{i<n}\chi(\sigma^i)\overline{\eta}^i=\sum_{i=0}^{i<n}\operatorname{Re}(\chi(\sigma^i)\overline{\eta}^i)\geq\chi(1)-\sum_{i=1}^{i<n}|\chi(\sigma^i)|>0.$$
This implies that the minimal polynomial of $\rho(\sigma)$ equals $x^n-1$.
\end{proof}

\section{Proof of Proposition~\ref{pr:1}}
In this section we prove Proposition~\ref{pr:1}.
The proof is divided into two lemmas.
\begin{lemma} For an odd integer $n\geq5$, consider $\sigma\in S_n$ of shape $[(n-2)^12^1]$ and an irreducible representation $\rho:S_n\rightarrow GL(V)$ of $S_n$ corresponding to the partition $(2^2,1^{n-4})$. 
Then $\sigma$ lacks nontrivial fixed vectors in $V$.
\end{lemma}
\begin{proof}
Put $\lambda=(2^2,1^{n-4})$ and denote by $\chi^\lambda$ the character of $\rho$.
Proposition~\ref{l:rep_Sn}$(iii)$ implies that $\chi^\lambda(\tau)=\operatorname{sgn}(\tau)\chi^{(n-2,2)}(\tau)$ for every 
$\tau\in S_n$. Therefore, applying Lemma~\ref{l:CharOfSmallDegree}, we can find $\chi^\lambda(\sigma^j)$ for every $0\leq j<2(n-2)$:
$$\chi^{\lambda}(\sigma^j)=\begin{cases} \frac{1}{2}(n-1)(n-2)-1=\frac{n(n-3)}{2} & \mbox{if } j=0,  \\ 
-1 & \mbox{if } j\neq0 \mbox{ is even}, \\
-\frac{1}{2}(n-3)(n-4) & \mbox{if } j=n-2, \\
-1 & \mbox{if } j\neq{n-2} \mbox{ and } j \mbox{ is odd}.  \end{cases}$$

Since among the numbers $0,\ldots,2(n-2)-1$ exactly $n-2$ are even and $n-2$ are odd, 
we infer that $|\sigma|\cdot\operatorname{dim}C_V(\rho(\sigma))=\sum_{j=0}^{j<|\sigma|}\chi^\lambda(\sigma^j)=\frac{n(n-3)}{2}-(n-3)-\frac{(n-3)(n-4)}{2}-(n-3)=~0$.
\end{proof}

\begin{lemma} Suppose that the pair $(\lambda,\sigma)$ is one of the following.
	\begin{enumerate}[(i)]
	\item $\lambda=(1^n)$ and $\sigma$ is odd,
        \item or $\sigma$ is $n$-cycle and either $\lambda=(n-1,1)$ or $n$ is odd and $\lambda=(2,1^{n-2}),$
        \item or $((2^3), [3^12^11^1])$, $((4^2), [5^13^1])$, $((2^4), [5^13^1])$, $((2^5), [5^13^12^1]).$
\end{enumerate}
	Then $\sigma$ has no nontrivial fixed vectors in $V$.
\end{lemma}
\begin{proof}
The claims corresponding to $(i)$ and $(ii)$ are considered in Theorem~\ref{th:1}. The last claim can be verified 
using characters tables of $S_6$, $S_8$, and $S_{10}$ \cite{Atlas} or GAP.
\end{proof}

\Addresses

\begin{thebibliography}{9}
  \bibitem{HallHig} P.\,Hall and G.\,Higman, On the $p$-length of $p$-soluble groups and reduction theorems for Burnside’s problem, Proc. London Math. Soc., {\bf6} (1956), 1--42.
  	
  \bibitem{TiepZal08} Pham Huu Tiep, A.E.\,Zalesskii, Hall-Higman type theorems for semisimple elements of finite classical groups, Proc. London Math. Soc., {\bf97}:3 (2008), 623--668.
	
  \bibitem{Zal96} A.E.\,Zalesskii, \textit{Eigenvalues of prime-order elements in projective representations of alternating groups.} (Russian) Vestsī Akad. Navuk Belarusī Ser. Fīz.-Mat. Navuk,  {\bf 131}:3 (1996), 41--43.
  
  \bibitem{KleshZal04} A.S.\,Kleshchev and A.E.\,Zalesski. \textit{Minimal polynomials of elements of order $p$ in  $p$-modular projective representations of alternating groups}, Proc. Amer. Math. Soc., {\bf 132}:6 (2004), 1605--1612.
   
  \bibitem{GAP} The GAP Group, GAP -- Groups, Algorithms, and Programming, Version 4.10.2; 2019. (\url{https://www.gap-system.org}) 
  
  \bibitem{James} G.\,James and A.\,Kerber, \textit{The representation theory of the symmetric group} (Encyclopedia of Mathematics and its Applications, Vol. 16, Addison-Wesley, Reading, Mass., 1981), pp. 510.

  \bibitem{FulHar} W.\,Fulton, J.\,Harris, \textit{Representation theory: a first course. Graduate texts in mathematics}. Springer-Verlag, New York, Berlin, Paris, 1991, 551 p.

  \bibitem{FominLulov} S.V.\,Fomin, N.\,Lulov, \textit{On the number of rim hook tableaux},
 J. Math. Sci. (N.Y.) {\bf87} (1997), 4118--4123.

  \bibitem{Rob55} H.\,Robbins, \textit{A remark on Stirling's formula}, Amer. Math. Monthly, {\bf62} (1955), 26--29.

  \bibitem{Rasala} R.\,Rasala, \textit{On the minimal degrees of characters of $S_n$}, J. Algebra {\bf45}:1 (1977),  132--181.

  \bibitem{Tong} Tong-Viet, Hung\,P, \textit{Symmetric groups are determined by their character degrees.}, J. Algebra, {\bf334} (2011), 275--284.

  

   \bibitem{Atlas} J.H.\,Conway, R.T.\,Curtis, S.P.\,Norton, R.A.\,Parker, R.A.\,Wilson, \textit{Atlas of finite groups}, Clarendon Press, Oxford, 1985.
\end{thebibliography}
\end{document}